\documentclass{amsart}
\usepackage{amsmath,amssymb,dsfont,amsthm}
\usepackage{xypic}
\usepackage{latexsym}
\usepackage{amsmath}
\usepackage{amsfonts}
\usepackage{geometry}
\usepackage{amssymb}
\usepackage{tensor}

\usepackage{graphicx}
\usepackage[cmtip,all]{xy}

\newtheorem{theorem}{Theorem}[section]
\newtheorem{lemma}[theorem]{Lemma}

\theoremstyle{definition}
\newtheorem{definition}[theorem]{Definition}

\theoremstyle{remark}

\newtheorem{corollary}[theorem]{Corollary}

\newtheorem{problem}[theorem]{Question}
\newtheorem{proposition}[theorem]{Proposition}

\numberwithin{equation}{section}

\begin{document}

\title[Deformations of Hyperbolic Cone-structures]{Deformations of Hyperbolic Cone-structures: Study of the collapsing case}


\author{Alexandre Paiva Barreto}
\address{Universidade Federal de S\~{a}o Carlos}
\email{alexandre@dm.ufscar.br}
\thanks{This work is part of the doctoral thesis of the author, made
  at Universit Paul Sabatier under the supervision of Professor Michel
Boileau and supported by CAPES and FAPESP}

\subjclass[2000]{Primary 55m20}

\date{}

\dedicatory{Dedicated to professor Paulo Sabini (1972 - 2007)}

\begin{abstract}
This work is devoted to the study of deformations of hyperbolic cone
structures under the assumption that the lengths of the singularity remain
uniformly bounded over the deformation. Given a sequence $\left(  M_{i}%
,p_{i}\right)  $ of pointed hyperbolic cone-manifolds with topological type
$\left(  M,\Sigma\right)  $, where $M$ is a closed, orientable and irreducible
$3$-manifold and $\Sigma$ an embedded link in $M$. If the sequence $M_{i}$
collapses and assuming that the lengths of the singularity remain uniformly
bounded, we prove that $M$ is either a Seifert fibered or a $Sol$ manifold. We
apply this result to a question stated by Thurston and to the study of
convergent sequences of holonomies.

\end{abstract}

\maketitle


\section{Introduction}

\qquad Fixed a closed, orientable and irreducible $3$-manifold $M$, this
text focus deformations of hyperbolic cone structures on $M$ which are
singular along a fixed embedded link $\Sigma=\Sigma_{1}\sqcup\ldots\sqcup%
\Sigma_{l}$ in $M$. A \textit{hyperbolic cone structure} with topological
type $\left( M,\Sigma\right) $ is a complete intrinsic metric on $M$ (see
section $2$ for the definition) such that every \textit{non-singular point}
(i.e. every point in $M-\Sigma$) has a neighborhood isometric to an open set
of $\mathbb{H}^{3}$, the hyperbolic space of dimension $3$, and that every 
\textit{singular point} (i.e. every point in $\Sigma$) has a neighborhood
isometric to an open neighborhood of a singular point of $\mathbb{H}%
^{3}\left( \alpha\right) $, the space obtained by identifying the sides of a
dihedral of angle $\alpha \in\left( 0,2\pi\right] $ in $\mathbb{H}^{3}$ by a
rotation about the axe of the dihedral. The angles $\alpha$ are called 
\textit{cone angles} and they may vary from one connected component of $%
\Sigma$ to the other. By convention, the complete structure on $M-\Sigma$ is
 considered as a hyperbolic cone structure with topological
type $\left( M,\Sigma\right) $ and cone angles equal to zero.

Unlike hyperbolic structures, which are rigid after Mostow, the hyperbolic
cone structures can be deformed (cf. \cite{HK2}). The difficulty to
understand these deformations lies in the possibility of degenerating the
structure. In other words, the Hausdorff-Gromov limit of the deformation
(see section 2 for the definition) is only an Alexandrov space which may
have dimension strictly smaller than $3$, although its curvature remains
bounded from below by $-1$ (cf. \cite{Koj}). In fact, the works of Kojima,
Hodgson-Kerckhoff and Fuji (see \cite{Koj}, \cite{HK} and \cite{Fuj}) show
that the degeneration of the hyperbolic cone structures occurs if and only
if the singular link of these structures intersects itself over the
deformation.

When $M$ is a hyperbolic manifold and $\Sigma $ is a finite union of simple
closed geodesics of $M$, it is know that $M-\Sigma $ admits a complete
hyperbolic structure (see \cite{Koj}). We are interested in studying the
 following question that was raised by W.Thuston in 80's:

\begin{problem}
\label{conjectura thurston}Let $M$ be a closed and orientable hyperbolic
manifold and suppose the existence of a simple closed geodesic $\Sigma$ in $M
$. Can the hyperbolic structure of $M$ be deformed to the complete hyperbolic
structure on $M-\Sigma$ through a path $M_{\alpha}$ of hyperbolic cone structures
with topological type $\left( M,\Sigma\right) $ and parametrized by the cone
 angles $\alpha \in\left[ 0,2\pi\right] $?
\end{problem}

When the deformation proposed by Thurston exists, it is a consequence of
Thurston's hyperbolic Dehn surgery Theorem that the length of $\Sigma $ must
converge to zero. In particular, we have that the length of the singular
link remains uniformly bounded over the deformation. Motivated by this
fact, we are interested in studying the deformations of hyperbolic cone structures
 with this aditional hypothesis on the length of the singular link.  This
 assumption is automatically verified when the holonomy representations of
 the hyperbolic cone structures are convergent. This happens (cf. \cite{CS}),
 for example, when $\Sigma $ is a small link in $M$. Note that this additional
 hypothesis avoids the undesirable case where the singularity becomes dense in
 the limiting Alexandrov space. 

In this paper, $M_{i}$ will always denote a sequence of hyperbolic
cone-manifolds with topological type $\left( M,\Sigma\right) $. Given points 
$p_{i}\in M_{i}$, suppose that the sequence $\left( M_{i},p_{i}\right) $
converges (in the Hausdorff-Gromov sense) to a pointed Alexandrov space $%
\left( Z,z_{0}\right) $. Our goal is to understand the metric properties of
the Alexandrov space $Z$ and exploit them to obtain further topological
information on the manifold $M$. To do this, we need the important notion of
collapse for a sequence and a continuous deformation of hyperbolic
cone-manifolds.

\begin{definition}
We say that a sequence $M_{i}$ of hyperbolic cone-manifolds with topological
type $\left( M,\Sigma\right) $ collapses if, for every sequence of points $%
p_{i}\in M-\Sigma$, the sequence $r_{inj}^{M_{i}-\Sigma}\left( p_{i}\right) $
consisting of their Riemannian injectivity radii in $M_{i}-\Sigma$ converges
to zero. Otherwise, we say that the sequence $M_{i}$ does not collapse.
\end{definition}

When a convergent sequence of hyperbolic cone-manifolds does not collapse,
the limit Alexandrov space must have dimension 3. In this case, several
geometric techniques are known to study the topological type of $M$. On the
collapsing case, however, most of the geometric information can be lost.
This happens because the dimension of the limit Alexandrov space may be
strictly smaller than 3 (see \ref{effondre a dist borne implica dim menor}).
Therefore, finding conditions that eliminate the possibility of collapse is
very useful.

Denote by $\mathcal{L}_{M_{i}}\left( \Sigma_{j}\right) $ the length of the
connected component $\Sigma_{j}$ of $\Sigma$ in the hyperbolic cone-manifold 
$M_{i}$. Using this notation, the principal result of this paper is the
following one:

\begin{theorem}
\label{teo principal}\label{teo principal introducao}Let $M$ be a closed,
orientable and irreducible $3$-manifold and let $\Sigma =\Sigma _{1}\sqcup
\ldots \sqcup \Sigma _{l}$ be an embedded link in $M$. Suppose the existence
of a sequence $M_{i}$ of hyperbolic cone-manifolds with topological type $%
\left( M,\Sigma \right) $ and having cone angles $\alpha _{ij}\in \left(
0,2\pi \right] $. If%
\begin{equation}
\sup \left\{ \mathcal{L}_{M_{i}}\left( \Sigma _{j}\right) \;;\;i\in \mathbb{N%
}\text{ and }j\in \left\{ 1,\ldots ,l\right\} \right\} <\infty 
\label{controledocomprimento}
\end{equation}%
and the sequence $M_{i}$ collapses, then $M$ is Seifert fibered or a $Sol$
manifold.
\end{theorem}

As a consequence of the theorem (\ref{teo principal introducao}), we obtain
 the following result related to the Thurston's question
(\ref{conjectura thurston}).

\begin{corollary}
\label{corolario da conjectura}Let $M$ be a closed and orientable hyperbolic 
$3$-manifold and suppose the existence of a finite union of simple closed
geodesics $\Sigma $ in $M$. Let $M_{\alpha }$ be a (angle decreasing)
deformation of this structure along a continuous path of hyperbolic
cone-structures with topological type $\left( M,\Sigma \right) $ and having the same
cone angle $\alpha \in \left( \theta,2\pi \right] \subset \left[ 0,2\pi \right] $
for all components of $\Sigma $. If%
\begin{equation}
\sup \left\{ \mathcal{L}_{M_{\alpha }}\left( \Sigma _{j}\right) \;;\;\alpha
\in \left( \theta,2\pi \right] \text{ and }j\in \left\{ 1,\ldots ,l\right\}
\right\} <\infty \text{,}  \label{controlecomprimento2}
\end{equation}%
then every convergent (in the Hausdorff-Gromov sense) sequence $M_{\alpha
_{i}}$, with $\alpha _{i}$ converging to $\theta$, does not collapses.
\end{corollary}

If the deformation in the statement of the corollary is supposed to be maximal,
 we remind that $\theta =0$ if and only if $\theta\leq \pi $. This is a consequence of
 Kojima's work in \cite{Koj}.

\section{Metric Geometry}

\qquad Given a metric space $Z$, the metric on $Z$ will always be denoted by
$d_{Z}\left(  \cdot,\cdot\right)  $. The open ball of radius $r>0$ about a
subset $A$ of $Z$ is going to be denoted by%
\[
B_{Z}\left(  A,r\right)  =%
{\textstyle\bigcup\limits_{a\in A}}
\left\{  z\in Z\;;\;d_{Z}\left(  z,a\right)  <r\right\}  \text{.}%
\]
A metric space $Z$ is called a \textit{length space} (and its metric is called
intrinsic) when the distance between every pair of points in $Z$ is given by
the infimum of the lengths of all rectificable curves connecting them. When a
minimizing geodesic between every pair of points exists, we say that $Z$ is
\textit{complete}.

For all $k\in\mathbb{R}$, denote $\mathbb{M}_{k}^{2}$ the complete and simply
connected two dimensional Riemannian manifold of constant sectional curvature
equal to $k$. Given a triple of points $\left(  x\;;\;y,z\right)  $ of $Z$, a
\textit{comparison triangle} for the triple is nothing but a geodesic triangle
$\Delta_{k}\left(  \overline{x},\overline{y},\overline{z}\right)  $ in
$\mathbb{M}_{k}^{2}$ with vertices $\overline{x}$, $\overline{y}$ and
$\overline{z}$ such that%
\[
d_{\mathbb{M}_{k}^{2}}\left(  \overline{x},\overline{y}\right)  =d_{Z}\left(
x,y\right)  \;\text{,}\;\;d_{\mathbb{M}_{k}^{2}}\left(  \overline{y}%
,\overline{z}\right)  =d_{Z}\left(  y,z\right)  \;\;\text{and}%
\;\;d_{\mathbb{M}_{k}^{2}}\left(  \overline{z},\overline{x}\right)
=d_{Z}\left(  z,x\right)  \text{.}%
\]
Note that a comparison triangle always exists when $k\leq0$. The
$k$\textit{-angle} of the triple $\left(  x\;;\;y,z\right)  $ is, by
definition, the angle $\measuredangle_{k}\left(  x\;;\;y,z\right)  $ of a
comparison triangle $\Delta_{k}\left(  \overline{x},\overline{y},\overline
{z}\right)  $ at the vertex $\overline{x}$ (assuming the triangle exists).

\begin{definition}
A finite dimensional (in the Hausdorff sense) length space $Z$ is called an
Alexandrov space of curvature not smaller than $k\in\mathbb{R}$ if every point
has a neighborhood $U$ such that, for all points $x,y,z\in U$, the angles
$\measuredangle_{k}\left(  x\;;\;y,z\right)  $, $\measuredangle_{k}\left(
y\;;\;x,z\right)  $ and $\measuredangle_{k}\left(  z\;;\;x,y\right)  $ are
well defined and satisfy%
\[
\measuredangle_{k}\left(  x\;;\;y,z\right)  +\measuredangle_{k}\left(
y\;;\;x,z\right)  +\measuredangle_{k}\left(  z\;;\;x,y\right)  \leq
2\pi\text{.}%
\]

\end{definition}

We point out that every hyperbolic cone-manifold is an Alexandrov space of
curvature not smaller than $-1$.

Suppose from now on that $Z$ is a $n$ dimensional Alexandrov space of
curvature not smaller than $k\in\mathbb{R}$. Consider $z\in Z$ and $\lambda
\in\left(  0,\pi\right)  $. The point $z$ is said to be $\lambda
$\textit{-strained} if there exists a set $\left\{  \left(  a_{i}%
,b_{i}\right)  \in Z\times Z\;;\;i\in\left\{  1,\ldots,n\right\}  \right\}  $,
called a $\lambda$-strainer at $z$, such that $\measuredangle_{k}\left(
x\;;\;a_{i},b_{i}\right)  >\pi-\lambda$ and%
\[
\max\left\{  \left\vert \measuredangle_{k}\left(  x\;;\;a_{i},a_{j}\right)
-\frac{\pi}{2}\right\vert ,\;\left\vert \measuredangle_{k}\left(
x\;;\;b_{i},b_{j}\right)  -\frac{\pi}{2}\right\vert ,\left\vert \measuredangle
_{k}\left(  x\;;\;a_{i},b_{j}\right)  -\frac{\pi}{2}\right\vert \right\}
<\lambda
\]
for all $i\neq j\in\left\{  1,\ldots,n\right\}  $. The set $R_{\delta}\left(
Z\right)  $ of $\lambda$-strained points of $Z$ is called the \textit{set of
}$\lambda$\textit{-regular points of }$Z$. It is a remarkable fact that
$R_{\delta}\left(  Z\right)  $ is an open and dense subset of $Z$.

Recall now, the notion of (pointed) Hausdorff-Gromov convergence (see
\cite{BBI}):

\begin{definition}
Let $\left(  Z_{i},z_{i}\right)  $ be a sequence of pointed metric spaces. We
say that the sequence $\left(  Z_{i},z_{i}\right)  $ converges in the
(pointed) Hausdorff-Gromov sense to a pointed metric space $\left(
Z,z_{0}\right)  $, if the following holds: For every $r>\varepsilon>0$, there
exist $i_{0}\in\mathbb{N}$ and a sequence of (may be non continuous) maps
$f_{i}:B_{Z_{i}}\left(  z_{i},r\right)  \rightarrow Z$ ($i>i_{0}$) such that

\begin{enumerate}
\item[i.] $f_{i}\left(  z_{i}\right)  =z_{0}$,

\item[ii.] $\sup\left\{  d_{Z^{\prime}}\left(  f_{i}\left(  z_{1}\right)
,f_{i}\left(  z_{2}\right)  \right)  -d_{Z}\left(  z_{1},z_{2}\right)
\;;\;z_{1},z_{2}\in Z\right\}  <\varepsilon$,

\item[iii.] $B_{Z}\left(  z_{0},r-\varepsilon\right)  \subset B_{Z}\left(
f_{i}\left(  B_{Z_{i}}\left(  z_{i},r\right)  \right)  ,\varepsilon\right)  $,

\item[iv.] $f_{i}\left(  B_{Z_{i}}\left(  z_{i},r\right)  \right)  \subset
B_{Z}\left(  z_{0},r+\varepsilon\right)  $.
\end{enumerate}
\end{definition}

Its a fundamental fact that the class of Alexandrov spaces of curvature not
smaller than $k\in\mathbb{R}$ is pre-compact with respect to the notion of
convergence in the Hausdorff-Gromov sense. In particular, every pointed
sequence of hyperbolic cone-manifolds with constant topological type has a
subsequence converging (in the Hausdorff-Gromov sense) to a pointed Alexandrov
space which may have dimension strictly smaller than $3$, although its
curvature remains bounded from below by $-1$.

The Fibration Theorem is an useful tool for the study of sequences of
Alexandrov spaces. The first version of this theorem, due to Fukaya (see
\cite{Fuk}), concerns the Riemannian manifolds with pinched curvature. Later
this theorem was extended by Takao Yamaguchi to Alexandrov Spaces (see
\cite{Yam}, \cite{SY} and \cite{Bel}). For the convenience of the reader, we
present here the precise statement of the Yamaguchi theorem that is going to be
needed in this work (see \cite{Bar}):

\begin{theorem}
[Yamaguchi]\label{teo fibracao}Let $\left(  M_{i},p_{i}\right)  $ be a
sequence of pointed complete Riemannian manifolds of dimension $3$ with
sectional curvature bounded from below by $-1$. Suppose that the sequence
$\left(  M_{i},p_{i}\right)  $ converges in the Hausdorff-Gromov sense to a
pointed complete length space $\left(  Z,z\right)  $ of dimension $2$ (resp.
dimension $1$), then there exists a constant $\lambda>0$ satisfying the
following condition: For every compact domain $Y$ of $Z$ contained in
$R_{\lambda}\left(  Z\right)  $ and for sufficiently large $i_{0}=i_{0}\left(
Y\right)  \in\mathbb{N}$ , there exist

\begin{enumerate}
\item[$\bullet$] a sequence $\tau_{i}>0$ ($i>i_{o}$) converging to zero,

\item[$\bullet$] a sequence $\mathcal{N}_{i}$ ($i>i_{o}$) of compact
$3$-submanifolds of $M_{i}$ (perhaps with boundary)

\item[$\bullet$] a sequence of $\tau_{i}$-approximations $\mathfrak{p}%
:\mathcal{N}_{i}\rightarrow Y$ which induces a structure of locally trivial
fibre bundle on $\mathcal{N}_{i}$. Furthermore, the fibers of this fibration
are circles (resp. spheres, tori).
\end{enumerate}
\end{theorem}

\section{Sequences of Hyperbolic cone-manifolds}

\qquad Recall that $M$ denotes a closed, orientable and irreducible
differential manifold of dimension $3$ and that $\Sigma=\Sigma_{1}\sqcup
\ldots\sqcup\Sigma_{l}$ denotes an embedded link in $M$. A sequence of
hyperbolic cone-manifolds with topological type $\left(  M,\Sigma\right)  $
will always be denoted by $M_{i}$.

Given a sequence $M_{i}$ as above, fix indices $i\in\mathbb{N}$ and
$j\in\left\{  1,\ldots,l\right\}  $. For sufficiently small radius $R>0$, the
metric neighborhood%
\[
B_{M_{i}}\left(  \Sigma_{j},R\right)  =\left\{  x\in M_{i}\;;\;d_{M_{i}%
}\left(  x,\Sigma_{j}\right)  <R\right\}
\]
of $\Sigma$ is a solid torus embedded in $M_{i}$. The supremum of the radius
$R>0$ satisfying the above property will be called \textit{normal injectivity
radius of }$\Sigma_{j}$\textit{\ in }$M_{i}$ and it is going to be denoted by
$R_{i}\left(  \Sigma_{j}\right)  $. Analogously we can define $R_{i}\left(
\Sigma\right)  $, the \textit{normal injectivity radius of }$\Sigma$. It is a
remarkable fact (see \cite{Fuj} and \cite{HK}) that the existence of a uniform
lower bound for $R_{i}\left(  \Sigma\right)  $ ensures the existence of a
sequence of points $p_{i_{k}}\in M$ such that the sequence $\left(  M_{i_{k}%
},p_{i_{k}}\right)  $ converges in the Hausdorff-Gromov sense to a pointed
hyperbolic cone-manifold $\left(  M_{\infty},p_{\infty}\right)  $ with
topological type $\left(  M,\Sigma\right)  $. Moreover, $M_{\infty}$ must be
compact provides that cone angles of $M_{i_{k}}$ are uniformly bounded from below.

The purpose of this section is to prove the Theorem (\ref{teo principal}).
This section is divided into two parts. The first part contains some
preliminary results whereas the remaining one presents the proof of the theorem.

Let us point out that, throughout the rest of the paper, the term "component"
is going to stand for "connected component"

\subsection{Preliminary results}

Let us begin with three elementary lemmas which will be important for the
proof of Theorem (\ref{teo principal}).

\begin{lemma}
\label{Classificacao de Toros}Suppose that $M-\Sigma$ is hyperbolic and let
$T$ be a two dimensional torus embedded in $M-\Sigma$. Then $T$ separates $M
$. Moreover, one and only one of the following statements holds:

\begin{enumerate}
\item[i.] $T$ is parallel to a component of $\Sigma$ (hence it bounds a solid
torus in $M$),

\item[ii.] $T$ is not parallel to a component of $\Sigma$ and it bounds a
solid torus in $M-\Sigma$,

\item[iii.] $T$ is not parallel to a component of $\Sigma$ and it is contained
in a ball $B$ of $M-\Sigma$. Furthermore, $T$ bounds a region in $B$ which is
homeomorphic to the exterior of a knot in $S^{3}$.
\end{enumerate}
\end{lemma}

\begin{proof}
Suppose that $T$ is not parallel to a connected component of $\Sigma$. Since
$M-\Sigma$ is hyperbolic, $T$ is compressible in $M-\Sigma$ and hence it
splits $M$ into two components, say $U$ and $V$. Denote by $\Sigma_{U}$ and by
$\Sigma_{V}$ the subsets of $\Sigma$ contained respectively in $U$ and in $V$.
Without loss of generality suppose $\Sigma_{U}\neq\emptyset$.

Let $D$ be a compression disk for $T$. Since $M-\Sigma$ is irreducible and
$\Sigma_{U}\neq\emptyset$, the sphere obtained by the compression of $T$ along
$D$ bounds a ball $B\subset M-\Sigma$ in the opposite side of $\Sigma_{U}$. In
particular, this implies that $\Sigma_{V}=\emptyset$, since $\Sigma_{U}$ and
$\Sigma_{V}$ lie in different sides of $S$, by construction. Thus $V$ must
satisfy one of the assertions (ii) or (iii) according to whether $D\subset V$
or $D\subset U$.\hfill
\end{proof}

\begin{lemma}
\label{troca exterior de no por toro}Given a closed ball $B\subset M-\Sigma$
and an embedded two dimensional torus $T\subset B$, let $U$ and $W$ be the
compact three manifolds (with toral boundary) obtained by surgery of $M$ along
$T$ (suppose $W\subset B$). If $W$ is homeomorphic to the exterior of a knot
in $S^{3}$, then we can replace $W$ by a solid torus $V$ without changing the
topological type of $M$.
\end{lemma}

\begin{proof}
Let $B^{\prime}$ be a closed ball and $h:\partial B\rightarrow\partial
B^{\prime}$ a diffeomorphism such that $B\cup_{h}B^{\prime}$, the manifold
obtained by gluing $B$ and $B^{\prime}$ through $h$, is diffeomorphic to
$S^{3}$. Note that the manifold $V^{\prime}=\left(  U\cap B\right)  \cup
_{h}B^{\prime}$ is a closed solid torus. Let $V$ be a closed solid torus and
$g:\partial V\rightarrow\partial V^{\prime}$ a diffeomorphism such that
$V\cup_{g}V^{\prime}$, the manifold obtained by gluing $V$ and $V^{\prime}$
through $g$, is diffeomorphic to $S^{3}$. Thus the compact three manifold
$\mathcal{B}=\left(  V^{\prime}\cup_{g}V\right)  -int\left(  B^{\prime
}\right)  =\left(  U\cap B\right)  \cup_{g}V$ is diffeomorphic to a closed ball.

It is a standard result of low dimensional topology that the manifold%
\[
\left(  M-int\left(  B\right)  \right)  \cup_{Id}\mathcal{B}=\left(
M-W\right)  \cup_{g}V
\]
obtained by using $Id:\partial B\rightarrow\partial B$ as gluing map, is
diffeomorphic to $M$.\hfill
\end{proof}

\begin{lemma}
\label{salgueiro fica proximo}Consider $\varepsilon\in\left(  0,1\right)  $
and let $\mathcal{M}$ be a hyperbolic cone-manifold with topological type
$\left(  M,\Sigma\right)  $ together with a metric neighborhood $V$
of $\Sigma$ whose components are homeomorphic to solid tori. Denote by $g$ the
hyperbolic metric on $\mathcal{M}-\Sigma$ and let $h$ be a Riemannian metric
on $M$ such that:

\begin{enumerate}
\item[i.] $h$ coincides with $g$ on $\mathcal{M}-V$,

\item[ii.] $h$ is a $\varepsilon$-perturbation of $g$ on $M-\Sigma$, i.e.%
\begin{equation}
\left\vert \left\Vert v\right\Vert _{g}-\left\Vert v\right\Vert _{h}%
\right\vert <\varepsilon\text{,} \label{norma proxima}%
\end{equation}
for every $x\in\mathcal{M}-\Sigma$ and $v\in K_{x}=\left\{  u\in T_{x}\left(
\mathcal{M}-\Sigma\right)  \;;\;\left\Vert u\right\Vert _{g}\leq2\right\}  $.
\end{enumerate}

\noindent Then the Hausdorff-Gromov distance between $\mathcal{M}$ and the
Riemannian manifold $\mathcal{N}=\left(  M,h\right)  $ is smaller than
$4\varepsilon D$, where $D=1+diam_{\mathcal{M}}\left(  M\right)  $.
\end{lemma}

\begin{proof}
To prove the lemma, it suffices to check that the identity map
$Id:\mathcal{M\rightarrow N}$ is a $2\varepsilon D$-isometry. Fix $x,y\in M$.
Since $\Sigma$ has codimension $2$, there exists a differentiable curve
$\alpha:\left[  0,L\right]  \rightarrow\mathcal{M}$ (parametrized by arc
length in $\mathcal{M}$) such that $\alpha\left(  \left(  0,L\right)  \right)
\subset\mathcal{M}-\Sigma$ and $L<d_{\mathcal{M}}\left(  x,y\right)
+\frac{\varepsilon}{2}$. Then%
\[%
\begin{array}
[c]{ll}%
d_{\mathcal{N}}\left(  x,y\right)  -d_{\mathcal{M}}\left(  x,y\right)  &
\displaystyle\leq\mathcal{L}_{\mathcal{N}}\left(  \alpha\right)
-\mathcal{L}_{\mathcal{M}}\left(  \alpha\right)  +\frac{\varepsilon}{2}\\
& \\
& \displaystyle\leq%
{\displaystyle\int_{0}^{L}}
\left\vert \left\Vert \alpha^{\prime}\left(  t\right)  \right\Vert
_{g}-\left\Vert \alpha^{\prime}\left(  t\right)  \right\Vert _{h}\right\vert
\,dt\,+\frac{\varepsilon}{2}\\
& \\
& \displaystyle<\varepsilon.d_{\mathcal{M}}\left(  x,y\right)  +\frac
{\varepsilon^{2}}{2}+\frac{\varepsilon}{2}\leq\varepsilon D
\end{array}
\]
and, in particular, $diam_{\mathcal{N}}\left(  M\right)  \leq2D$.

Similarly, there exists a differentiable curve $\beta:\left[  0,L^{\prime
}\right]  \rightarrow\mathcal{N}$ (parametrized by arc length in $\mathcal{N}%
$) such that $\beta\left(  \left(  0,L^{\prime}\right)  \right)
\subset\mathcal{N}-\Sigma$ and $L^{\prime}<d_{\mathcal{N}}\left(  x,y\right)
+\frac{\varepsilon}{2}$. Inequality (\ref{norma proxima}) implies that
$\left\{  u\in T_{z}\left(  \mathcal{M}-\Sigma\right)  \;;\;h_{z}\left(
u,u\right)  =1\right\}  \subset K_{z}$, for all $z\in\mathcal{N}-\Sigma$. So%
\[
\left\vert \left\Vert \beta^{\prime}\left(  t\right)  \right\Vert
_{g}-\left\Vert \beta^{\prime}\left(  t\right)  \right\Vert _{h}\right\vert
<\varepsilon\text{,}%
\]
for all $t\in\left[  0,L^{\prime}\right]  $. Analogously to the preceding case
it then follows that $d_{\mathcal{M}}\left(  x,y\right)  -d_{\mathcal{N}%
}\left(  x,y\right)  \leq2\varepsilon D$.\hfill
\end{proof}

\subsection{Proof of the Theorem (\ref{teo principal})}

\qquad The purpose of this section is to study a collapsing sequence $M_{i}$.
According to \cite[theorem 1]{Fuj}, the sequence $M_{i}$ cannot collapse when
$\lim\limits_{i\rightarrow\infty}R_{i}\left(  \Sigma_{j}\right)  =\infty$, for
all components $\Sigma_{j}$ of $\Sigma$. Given $p\in\Sigma_{1}$, we can assume
without loss of generality that $\sup\left\{  R_{i}\left(  \Sigma_{1}\right)
\;;\;i\in\mathbb{N}\right\}  <\infty$ and that the sequence $\left(
M_{i},p\right)  $ converges in the Hausdorff-Gromov sense to a pointed
Alexandrov space $\left(  Z,z_{0}\right)  $. We are interested in the case
where the length of the singularity remains uniformly bounded, i.e. where%
\[
\sup\left\{  \mathcal{L}_{M_{i}}\left(  \Sigma_{j}\right)  \;;\;i\in
\mathbb{N}\;\text{,}\;j\in\left\{  1,\ldots,l\right\}  \right\}
<\infty\text{.}%
\]
Again an Ascoli-type argument implies (passing to a subsequence if necessary)
that each component $\Sigma_{j}$ of $\Sigma$ satisfies one, and only one, of
the following statements:

\begin{enumerate}
\item $\sup\left\{  d_{M_{i}}\left(  p,\Sigma_{j}\right)  \;;\;i\in
\mathbb{N}\right\}  <\infty$ and $\Sigma_{j}$ converges in the
Hausdorff-Gromov sense to a closed curve $\Sigma_{j}^{Z}\subset Z$,

\item $\lim\limits_{i\in\mathbb{N}}d_{M_{i}}\left(  p,\Sigma_{j}\right)
=\infty$.
\end{enumerate}

\noindent This dichotomy allows us to write $\Sigma=\Sigma_{0}\sqcup
\Sigma_{\infty}$, where $\Sigma_{0}$ contains the components $\Sigma_{j}$ of
$\Sigma$ which satisfies item (1) (in particular $\Sigma_{1}\subset\Sigma_{0}%
$) and $\Sigma_{\infty}$ those that satisfies the item (2). Now consider the
compact set%
\[
\Sigma_{Z}=%
{\textstyle\bigcup\limits_{\Sigma_{j}\subset\Sigma_{0}}}
\Sigma_{j}^{Z}\subset Z\text{.}%
\]

The following proposition provides information on the dimension of the limit
of a sequence of hyperbolic cone-manifolds which collapses.

\begin{lemma}
\label{effondre a dist borne implica dim menor}Given a sequence $p_{i}\in M$,
suppose that the sequence $\left(  M_{i},p_{i}\right)  $ converges in the
Hausdorff-Gromov sense to a pointed Alexandrov space $\left(  Z,z_{0}\right)
$. If the sequence $M_{i}$ collapses and verifies%
\[
\sup\left\{  \mathcal{L}_{M_{i}}\left(  \Sigma_{j}\right)  \;;\;i\in
\mathbb{N}\;,\;j\in\left\{  1,\ldots,l\right\}  \right\}  <\infty\text{,}%
\]
then the dimension of $Z$ is strictly smaller than $3$.
\end{lemma}

\begin{proof}
We can assume that there exists $z_{\infty}\in Z-\Sigma_{Z}$, since the
equality $\Sigma_{Z}=Z$ can only happens when the dimension of $Z$ is strictly
smaller than $3$. Set $\delta=d_{Z}\left(  z_{\infty},\Sigma_{Z}\right)  >0$
and take a sequence $q_{i}\in M-\Sigma$ such that%
\[
\lim_{i\rightarrow\infty}q_{i}=z_{\infty}\qquad\text{and}\qquad B_{M_{i}%
}\left(  q_{i},\frac{\delta}{3}\right)  \cap B_{M_{i}}\left(  \Sigma
,\frac{\delta}{3}\right)  =\emptyset\text{.}%
\]
By definition, the sequence $\left(  M_{i},q_{i}\right)  $ converges in the
Hausdorff-Gromov sense to $\left(  Z,z_{\infty}\right)  $. Since the sequence
$M_{i}$ collapses, we have $\lim\limits_{i\rightarrow\infty}r_{inj}^{M_{i}%
}\left(  q_{i}\right)  =0$.

According to \cite[Lemma 3.8]{HK} we can change the metrics of $M_{i}$ on
$B_{M_{i}}\left(  \Sigma,\frac{\delta}{3}\right)  -\Sigma$ to obtain a
sequence of complete Riemannian manifolds $N_{i}$ homeomorphic to $M-\Sigma$
which have finite volume and pinched sectional curvature. Without loss of
generality, we can assume that the sequence $\left(  N_{i},q_{i}\right)  $
converges in the Hausdorff-Gromov sense to a pointed Alexandrov space $\left(
Y,y_{\infty}\right)  $. Note that, by construction, the balls $B_{N_{i}%
}\left(  q_{i},\frac{\delta}{3}\right)  $ and $B_{M_{i}}\left(  q_{i}%
,\frac{\delta}{3}\right)  $ are isometric. Then the balls $B_{Y}\left(
y_{\infty},\frac{\delta}{3}\right)  $ and $B_{Z}\left(  z_{\infty}%
,\frac{\delta}{3}\right)  $ are also isometric and we have $\lim
\limits_{i\rightarrow\infty}r_{inj}^{N_{i}}\left(  q_{i}\right)  =0$.

According to \cite[Corollary 8.39]{Gro}, $B_{Z}\left(  z_{\infty},\frac
{\delta}{3}\right)  $ has Hausdorff dimension strictly smaller than $3$. Given
that the dimension of an Alexandrov space is a well-defined integer, it
follows that $Z$ also has dimension strictly smaller than $3$.\hfill
\end{proof}

The preceding lemma says that the only possibilities for the dimension of $Z$
are $2$, $1$ or $0$. The proof of Theorem (\ref{teo principal}) will
accordingly be split in three cases.

\subsubsection{Case where $Z$ has dimension $2$}

It is a classical result of metric geometry that all two dimensional
Alexandrov spaces are topological manifolds of dimension $2$, possibly with
boundary. We can describe the boundary of $Z$ as%
\[
\partial Z=%
{\textstyle\bigsqcup\limits_{k\in\Gamma}}
\partial Z_{k}%
\]
where $\partial Z_{k}$ are the connected components of $\partial Z$. Note that
each component $\partial Z_{k}$ is homeomorphic to a circle or to a straight line.

\begin{theorem}
[collapsing - dimension 2]\label{effondrement en dim 2}Suppose the existence
of a point $p\in\Sigma_{1}$ such that the sequence $\left(  M_{i},p\right)  $
converges (in the Hausdorff-Gromov sense) to a two dimensional pointed
Alexandrov space $\left(  Z,z_{0}\right)  $ and also that $\sup\left\{
R_{i}\left(  \Sigma_{1}\right)  \;;\;i\in\mathbb{N}\right\}  <\infty$.
Assuming that%
\[
\sup\left\{  \mathcal{L}_{M_{i}}\left(  \Sigma_{j}\right)  \;;\;i\in
\mathbb{N}\;,\;j\in\left\{  1,\ldots,l\right\}  \right\}  <\infty\text{,}%
\]
it follows that:

\begin{itemize}
\item[i.] $M$ is Seifert fibered,

\item[ii.] if $\partial Z\neq\emptyset$, then $\partial Z$ has only one
connected component and $M$ is a lens space,

\item[iii.] if $Z$ is not compact, then $\lim\limits_{i\in\mathbb{N}}%
R_{i}\left(  \Sigma_{j}\right)  =\infty$, for any component $\Sigma_{j}$ of
$\Sigma_{\infty}$.
\end{itemize}
\end{theorem}

\begin{proof}
According to \cite{Sal}, the metric of hyperbolic cone-manifolds $M_{i}$ in
some small neighborhoods of $\Sigma$ can be deformed to yield a sequence of
complete Riemannian manifolds $N_{i}$ (homeomorphic to $M$) with sectional
curvature not smaller than $-1$. Furthermore thanks to Lemma
(\ref{salgueiro fica proximo}), we can suppose that the sequence $\left(
N_{i},p\right)  $ converges (in the Hausdorff-Gromov sense) to $\left(
Z,z_{0}\right)  $.

If $Z$ is compact, the statement follows directly from the results of Shioya
and Yamaguchi in \cite{SY}. To be more precise, when $Z$ does not have
boundary, \cite[Theorem 0.2]{SY} applied to the sequence $N_{i}$ gives that
$M$ is Seifert fibered. When $Z$ has boundary, \cite[Corollary 0.4]{SY}
applied to the sequence $N_{i}$ implies ($M$ is irreducible) that $Z$ is a
closed disk with at most one cone in its interior. Furthermore, \cite[Theorem
0.3]{SY} implies that $M$ is a lens space (in particular it is also Seifert fibered).

Suppose from now on that $Z$ is not compact. Since the normal injectivity radius
of components of $\Sigma_{0}$ are uniformly bounded, there exists $R>0$ such
that%
\[
B_{M_{i}}\left(  \Sigma_{j},R_{i}\left(  \Sigma_{j}\right)  \right)  \subset
B_{M_{i}}\left(  p,R/2\right)
\]
for all $i\in\mathbb{N}$ and for every component $\Sigma_{j}\subset\Sigma_{0}%
$. Let $K$ be a compact and connected two dimensional submanifold of $Z$ such
that $B_{Z}\left(  z_{0},R\right)  \subset K$ (therefore $\Sigma_{Z}\subset
K$), $\partial K$ is a disjoint union of circles and $Z-K$ is a disjoint union
of components of infinite diameter.

Set $\Lambda=\left\{  k\in\Gamma\;;\;\partial Z_{k}\cap K\neq\emptyset
\right\}  $. If $\partial Z\neq\emptyset$, we assume also that:

\begin{itemize}
\item[$\bullet$] $\Lambda\neq\emptyset$, that is $K\cap\partial Z\neq
\emptyset$,

\item[$\bullet$] $\partial Z_{k}\cap K=\partial Z_{k}$, for all $k\in\Lambda$
such that the component $\partial Z_{k}$ is compact,

\item[$\bullet$] $\partial Z_{k}\cap K$ is connected, for all $k\in\Lambda
$. 
\end{itemize}

For each $k\in\Lambda$, denote by $\partial K_{k}$ the (unique by
construction) connected component of $\partial K$ such that $\partial
Z_{k}\cap\partial K_{k}\neq\emptyset$. The boundary of $K$ is given by%
\[
\partial K=%
{\textstyle\bigsqcup\limits_{k\in\Lambda}}
\partial K_{k}\sqcup%
{\textstyle\bigsqcup\limits_{m\in\Gamma-\Lambda}}
\partial K_{m}\text{ ,}%
\]
where $\partial K_{m}$ are the components of $\partial K$ which does not
intersect the boundary of $Z$.

Let $\lambda$ be the constant given by the Fibration Theorem
(\ref{teo fibracao}). Since $K$ is compact, the sets of $\lambda$-cone-points%
\[
\mathcal{C}=\left\{  z\in K-\partial Z\;;\;\mathcal{L}\left(  T_{z}Z\right)
\leq2\pi-\lambda\right\}
\]%
\[
\mathcal{C}_{k}=\left\{  z\in\partial Z_{k}\cap K\;;\;\mathcal{L}\left(
T_{z}Z\right)  \leq\pi-\lambda\right\}  \subset\partial K_{k}\qquad\left(
k\in\Lambda\right)
\]
are finite (cf. \cite[Theorem 2.1]{SY}). Modulo increasing $K$, we can assume
that $C$ is contained in its interior. Let $s_{1}\gg s_{2}>0$ be such that

\begin{itemize}
\item[$\bullet$] $B_{K}\left[  z,s_{1}\right]  \subset int\left(  K\right)  $
and $B_{K}\left[  z,s_{1}\right]  $ is homeomorphic to $D^{2}$, for all
$z\in\mathcal{C}$,

\item[$\bullet$] $B_{K}\left[  z,s_{1}\right]  \cap B_{K}\left[  z^{\prime
},s_{1}\right]  =\emptyset$, for all $z,z^{\prime}\in\mathcal{C}\cup
\bigcup\limits_{k\in\Lambda}\mathcal{C}_{k}$,

\item[$\bullet$] For all $z\in\mathcal{C}\cup\bigcup\limits_{k\in\Lambda
}\mathcal{C}_{k}$ and for all $k\in\Lambda$, we have%
\[
B_{k}\left(  \partial K_{k},s_{2}\right)  \cap B_{K}\left(  z,s_{1}\right)
\neq\emptyset\qquad\text{if and only if}\qquad z\in\mathcal{C}_{k}\text{.}%
\]

\end{itemize}

Consider the compact subset $Y$ of $Z$ defined by%
\[
Y=K-%
{\textstyle\bigcup\limits_{z\in\mathcal{C}}}
B_{K}\left(  z,s_{1}\right)  -%
{\textstyle\bigsqcup\limits_{k\in\Lambda}}
U_{k}\subset R_{\lambda}\left(  Z\right)  \text{,}%
\]
where $U_{k}=B_{K}\left(  \partial Z_{k},s_{2}\right)  \cup%
{\textstyle\bigcup\limits_{z\in\mathcal{C}_{k}}}
B_{K}\left(  z,s_{1}\right)  $, for every $k\in\Lambda$. The boundary of $Y$
is given by%
\begin{equation}
\partial Y=%
{\textstyle\bigsqcup\limits_{k\in\Lambda}}
\partial Y_{k}\sqcup%
{\textstyle\bigsqcup\limits_{z\in\mathcal{C}}}
\partial Y_{z}\sqcup%
{\textstyle\bigsqcup\limits_{m\in\Gamma-\Lambda}}
\partial K_{m}\text{,} \label{bordo de Y}%
\end{equation}
where $\partial Y_{k}$ and $\partial Y_{z}$ are the components of $\partial Y
$ which intersect respectively the regions $U_{k}$ and $B_{K}\left(
z,s_{1}\right)  $.

Without loss of generality, the Fibration Theorem (\ref{teo fibracao}) gives
us a sequence $\mathcal{N}_{i}$ of compact and connected three dimensional
submanifolds of $M$, a sequence $\tau_{i}>0$ that converges to zero and a
sequence of $\tau_{i}$-approximations $\mathfrak{p}_{i}:\mathcal{N}%
_{i}\rightarrow Y$ which induces a structure of locally trivial circle bundle
on $\mathcal{N}_{i}$. Since the manifold $M$ is orientable, it follows that
the connected components of $\partial\mathcal{N}_{i}=\mathfrak{p}_{i}%
^{-1}\left(  \partial Y\right)  $ are two dimensional tori.

Let $C$ be a component of $\partial Y$\textit{\ }and $T_{i}$ the component of
$\partial\mathcal{N}_{i}$ associated with it ($T_{i}=\mathfrak{p}_{i}%
^{-1}\left(  C\right)  $). By construction, we can choose base points
$q_{i}\in T_{i}$ and $q\in C$ such that the pointed components $\left(
B_{i},q_{i}\right)  $ of $N_{i}-int\left(  \mathcal{N}_{i}\right)  $
associated with the torus $T_{i}$ ($\partial B_{i}=T_{i}$), converges in the
Hausdorff-Gromov sense to the pointed component $\left(  B,q\right)
$\textit{\ }of\textit{\ }$Z-int\left(  Y\right)  $ associated to $C$
($\partial B=C$)\textit{.}%

\begin{figure}
[h]
\begin{center}
\includegraphics[scale=0.6]{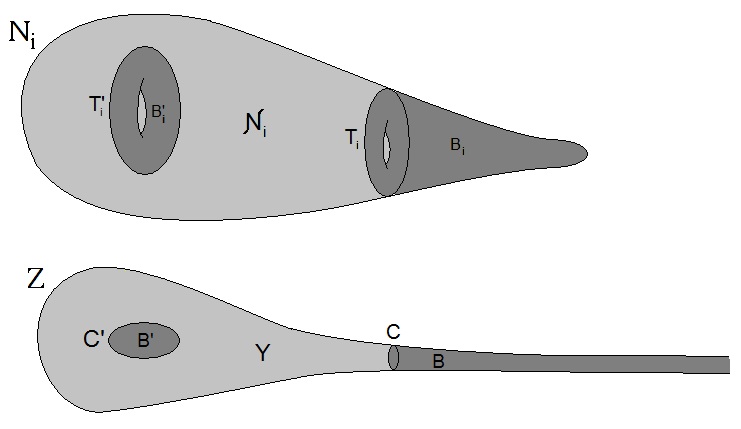}%
\caption{}%
\label{figura posta depois sem abusar}
\end{center}
\end{figure}

The above paragraph shows that the bijection between the boundary components
of $Z-Y$ and those of $N_{i}-\mathcal{N}_{i}$ induces a bijection between the
components of $Z-Y$ and those of $N_{i}-\mathcal{N}_{i}$. Moreover, the
diameter of the components of $N_{i}-\mathcal{N}_{i}$ associated with a
component $B$ of $Z-Y$ goes to infinity if and only if $B$ is noncompact.

Without loss of generality, we can suppose that $\Sigma_{\infty}\subset
N_{i}-\mathcal{N}_{i}$, for all $i\in\mathbb{N}$. It follows from lemma
(\ref{Classificacao de Toros}) that the singular components of $N_{i}%
-\mathcal{N}_{i}$ are solid tori whose souls are components of $\Sigma
_{\infty}$. Moreover, as the boundary components of $\mathcal{N}_{i}$ remain a
finite distance from the base point, the normal injectivity radii of the
components of $\Sigma_{\infty}$ became infinite with $i$. Note that this
implies that the singular components of $N_{i}-\mathcal{N}_{i}$ must be associated
with noncompact components of $Z-Y$.

\begin{lemma}
\label{lemma collapse dim 2}\textit{There exists }$i_{0}\in N$\textit{\ such
that, for every }$i>i_{0}$\textit{, the components of }$N_{i}-N_{i}%
$\textit{\ are homeomorphic to either a solid torus or to the exterior of a
knot in }$S^{3}$\textit{\ which, in addition, is contained in an embedded ball
in }$M-\Sigma$\textit{.}
\end{lemma}

\noindent\textbf{Proof of Lemma \ref{lemma collapse dim 2} : }A boundary torus
of $N_{i}-\mathcal{N}$ has 3 possible natures according to the decomposition
(\ref{bordo de Y}) of $\partial Y$. We shall consider separately each of there
possibilities. 

\bigskip
\noindent\textbf{1}$^{\text{\textbf{st}}}$ \textbf{type }: \textit{The torus
}$T_{i}^{z}=\mathfrak{p}_{i}^{-1}\left(  \partial Y_{z}\right)  \subset
\partial N_{i}$\textit{, }$z\in\mathcal{C}$\textit{.} 

Consider a point $z\in\mathcal{C}$. By rescaling the metric with respect to
the length of the fibers and using the stability theorem of Perelman (cf.
\cite{Kap}) , Shioya and Yamaguchi show in \cite[theorem 0.2]{SY} that the
region $B_{i}^{z}$ of $N_{i}-\mathcal{N}_{i}$ bounded by the torus $T_{i}^{z}
$ is a solid torus, for sufficiently large $i$. Furthermore, the solid torus
$B_{i}^{z}$ is glued on $\mathcal{N}_{i}$ without "killing" the fiber. 

\bigskip
\noindent\textbf{2}$^{\text{\textbf{nd}}}$ \textbf{type}: \textit{The torus
}$T_{i}^{k}=\mathfrak{p}_{i}^{-1}\left(  \partial Y_{k}\right)  \subset
\partial N_{i}$\textit{, }$k\in\Lambda$\textit{.} %
\begin{figure}
[h]
\begin{center}
\includegraphics[scale=0.5]{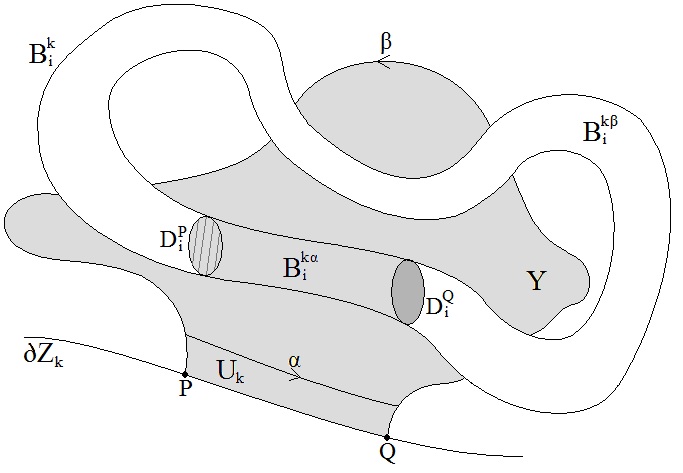}%
\caption{Tori of second type}%
\label{figura artigo}%
\end{center}
\end{figure}

Given an index $k\in\Lambda$, we can decompose the component $\partial Y_{k}$
in two simple arcs $\alpha$ and $\beta$ with the same ends, say $P$ and $Q$,
and such that $\alpha=$ $U_{k}\cap\partial Y_{k}$. Note that $\alpha=\partial
Y_{k}$ and $\beta$ is reduced to a single point (in particular $P=Q$) when
$\partial Z_{k}$ is compact.

Let $B_{i}^{k}$ be the region of $N_{i}-\mathcal{N}_{i}$ bounded by the torus
$T_{i}^{k}$. Using the same techniques, \cite[Theorem 0.3]{SY} shows that we
can choose, for sufficiently large $i$, a region $B_{i}^{k\alpha}$ inside
$B_{i}^{k}$ which is homeomorphic to $\alpha\times D^{2}$. Moreover:

\begin{itemize}
\item[$\bullet$] the sequence formed by the regions $B_{i}^{k\alpha}$
converges in the Hausdorff-Gromov sense to the region $U_{k}$,

\item[$\bullet$] the circle $\left\{  z\right\}  \times\partial D^{2}$ is a
fiber of the Seifert fibration on the boundary of $\mathcal{N}_{i}$, for every
$z\in\alpha$.
\end{itemize}

\noindent Then, for sufficiently large $i$, the regions $B_{i}^{k}$ are
homeomorphic to solid tori provided that $\partial Z_{k}$ is compact. Suppose
now that $\partial Z_{k}$ is not compact.

Let $D_{i}^{P}$ and $D_{i}^{Q}$ be the disks of $B_{i}^{k\alpha}$ which are
glued respectively on the fibers $\mathfrak{p}_{i}^{-1}\left(  P\right)  $ and
$\mathfrak{p}_{i}^{-1}\left(  Q\right)  $. The fibers above $\beta$ together
with the disks $D_{i}^{P}$ and $D_{i}^{Q}$ yield an embedded sphere in
$N_{i}-\Sigma$. This sphere splits $N_{i}$ into two regions: the region
$B_{i}^{k\beta}=B_{i}^{k}-int\left(  B_{i}^{k\alpha}\right)  $ and the region
$N_{i}-int\left(  B_{i}^{k\beta}\right)  $ which contain the singular set
$\Sigma_{0}\neq\emptyset$. Since the manifold $N_{i}-\Sigma$ is irreducible,
the region $B_{i}^{k\beta}$ must be homeomorphic to a ball and therefore, for
sufficiently large $i$, the region $B_{i}^{k}=B_{i}^{k\alpha}\cup
B_{i}^{k\beta}$ is homeomorphic to a solid torus. 

\bigskip
\noindent\textbf{3}$^{\text{\textbf{th}}}$ \textbf{type}: \textit{The torus
}$T_{i}^{m}=\mathfrak{p}_{i}^{-1}\left(  \partial K_{m}\right)  \subset
\partial N_{i}$\textit{, }$m\in\Gamma-\Lambda$\textit{.} 

Given an index $m\in\Gamma-\Lambda$, let $B_{i}^{m}$ be the region of
$N_{i}-\mathcal{N}_{i}$ bounded by the torus $T_{i}^{k}$. Since $B_{Z}\left(
z_{0},R\right)  \subset K$, the torus $T_{i}^{m}$ is contained in $M-\Sigma$
and cannot be parallel to a component of $\Sigma_{0}$. According to Lemma
(\ref{Classificacao de Toros}), the region $B_{i}^{m}$ is homeomorphic either
to a solid torus (possibly having a component of $\Sigma_{\infty}$ as its
soul) or to the exterior of a knot in $S^{3}$ which is contained in an
embedded ball in $M-\Sigma$.\hfill$\diamond$

\bigskip
Fix from now on a index $i>i_{0}$. According to Lemma
(\ref{troca exterior de no por toro}), the preceding claim shows that all
components of $N_{i}-\mathcal{N}_{i}$ can be supposed to be homeomorphic to
solid tori.

Let $\mathcal{W}_{i}$ be the three manifold obtained by the union of
$\mathcal{N}_{i}$ with all torus components of $N_{i}-\mathcal{N}_{i}$ which
are glued to the boundary of $\mathcal{N}_{i}$ without "killing" the fibers of
the Seifert fibration. For example, $B_{i}^{z}\subset\mathcal{W}_{i}$ and
$B_{i}^{k}\subset N_{i}-\mathcal{W}_{i}$, for all $z\in\mathcal{C}$ and
$k\in\Lambda$.

The Seifert fibration on $\mathcal{N}_{i}$ (with base $Y$) naturally extends
to a Seifert fibration on all off $\mathcal{W}_{i}$. Moreover, the underlying
space of  its basis is the topological surface $\mathcal{Y}$ obtained by
gluing disks on the components of $\partial Y$ associated with the components
of $\mathcal{W}_{i}-\mathcal{N}_{i}$. If $\mathcal{W}_{i}=N_{i}$ (so that
$\partial\mathcal{Y}=\emptyset$), the manifold $M$ is Seifert fibered and
$\partial Z=\emptyset$. Thus let us assume from now on that $\mathcal{W}%
_{i}\neq N_{i}$ (so that $\partial\mathcal{Y}\neq\emptyset$ and $\partial
\mathcal{W}_{i}\neq\emptyset$).

Suppose the existence of an essential arc $\gamma$ properly immersed in
$\mathcal{Y}$. Then the fibration above $\gamma$ provides an essential annulus
embedded in $\mathcal{W}_{i}$. Since all the components of $N_{i}%
-\mathcal{W}_{i}$ are solid torus that are glued to $\partial\mathcal{W}_{i}$
so as to "kill" the fibers of the Seifert fibration, this essential annulus
becomes an essential sphere in $N_{i}$. This is impossible since $M$ is
irreducible. Therefore $\mathcal{Y}$ has non-negative Euler characteristic, it
has exactly one boundary component and it contains at most one interior cone
point. In particular, it follows that $Z$ has at most one boundary component
and at most one interior cone point. In fact, if $Z$ has two or more interior
cone points (boundary components), then the compact $K$ can be chosen so as to
contains more than one interior cone points (boundary components) and this
would contradict to the non existence of a second interior cone point (boundary components)
in $\mathcal{Y}$.

Because $\mathcal{Y}$ is compact, the discussion above implies that
$\mathcal{Y}$ is homeomorphic to a disk $D^{2}$. So $\mathcal{W}_{i}$ is
homeomorphic to a solid torus and thus that $M$ is a lens space.\hfill
\end{proof}

\subsubsection{Case where $Z$ has dimension $1$}

It is also a standard result of metric geometry that all one dimensional
Alexandrov spaces are topological manifolds of dimension $1$. Precisely, they
are homeomorphic to on of the following models: $S^{1}$, $\left[  0,1\right]
$, $\left[  0,\infty\right)  $ or $\left(  -\infty,\infty\right)  $.

\begin{theorem}
[collapsing - dimension 1]\label{effondrement en dim 1}Suppose the existence
of a point $p\in\Sigma_{1}$ such that the sequence $\left(  M_{i},p\right)  $
converges (in the Hausdorff-Gromov sense) to a one dimensional pointed
Alexandrov space $\left(  Z,z_{0}\right)  $ and also that $\sup\left\{
R_{i}\left(  \Sigma_{1}\right)  \;;\;i\in\mathbb{N}\right\}  <\infty$.
Assuming that%
\[
\sup\left\{  \mathcal{L}_{M_{i}}\left(  \Sigma_{j}\right)  \;;\;i\in
\mathbb{N}\;,\;j\in\left\{  1,\ldots,l\right\}  \right\}  <\infty\text{,}%
\]
it follows that:

\begin{itemize}
\item[i.] $Z$ is homeomorphic to $S^{1}$ and $M$ is a euclidean, $Nil$ or
$Sol$ manifold,

\item[ii.] $Z$ is homeomorphic to $\left[  0,1\right]  $ and $M$ is a lens
space or a euclidean, $Nil$, $Sol$ or prism manifold,

\item[iii.] $Z$ is homeomorphic to $\left[  0,\infty\right)  $ or $\left(
-\infty,\infty\right)  $, $M$ is a lens space or prism manifold and
$\lim\limits_{i\in\mathbb{N}}R_{i}\left(  \Sigma_{j}\right)  =\infty$, for
every component $\Sigma_{j}$ of $\Sigma_{\infty}$.
\end{itemize}
\end{theorem}

\begin{proof}
If $Z$ is compact (i.e. homeomorphic to $S^{1}$ or $\left[  0,1\right]  )$,
then the assertions above follow directly from the results of Shioya and
Yamaguchi in \cite{SY}. To be more precise, when $Z$ is homeomorphic to
$S^{1}$, the irreducibility of $M$ implies that it is an euclidean, Nil or Sol
manifold (see remark before Theorem 0.5 and table 1 of \cite{SY}). When $Z$ is
homeomorphic to $\left[  0,1\right]  $, the irreducibility of $M$ implies that
it is a lens space or a euclidean, Nil, Sol or prism manifold (cf. \cite[table
1]{SY}).

Suppose from now on that $Z$ is not compact, i.e. that it is homeomorphic to
$\left[  0,\infty\right)  $ or $\left(  -\infty,\infty\right)  $. According to
\cite{Sal}, the metric of hyperbolic cone-manifolds $M_{i}$ on some small
neighborhoods of $\Sigma$ can be deformed to yield a sequence of complete
Riemannian manifolds $N_{i}$ (homeomorphic to $M$) with sectional curvature
not smaller than $-1$. Furthermore thanks to Lemma
(\ref{salgueiro fica proximo}), we can suppose that the sequence $\left(
N_{i},p\right)  $ converges (in the Hausdorff-Gromov sense) to $\left(
Z,z_{0}\right)  $.

Since $\sup\left\{  R_{i}\left(  \Sigma_{j}\right)  \;;\;i\in\mathbb{N}%
\;\text{,}\;\Sigma_{j}\in\Sigma_{0}\right\}  <\infty$, there exists $R>0$ such
that%
\[
B_{M_{i}}\left(  \Sigma_{j},R_{i}\left(  \Sigma_{j}\right)  \right)  \subset
B_{M_{i}}\left(  p,R/2\right)
\]
for all $i\in\mathbb{N}$ and for every component $\Sigma_{j}$ of $\Sigma_{0}$.
Consider a compact $K$ of $Z$ which is homeomorphic to $\left[  0,1\right]  $
and contains $B_{Z}\left(  z_{0},R\right)  \cup$ $\partial Z$.

Let $\lambda$ be the constant given by the Fibration Theorem
(\ref{teo fibracao}) and set $Y=K-U$, where $U$ is empty or a small open
neighborhood of $\partial Z$ according to whether or not the boundary of $Z$
is empty.

Without loss of generality, the Fibration Theorem (\ref{teo fibracao}) gives
us a sequence $\mathcal{N}_{i}$ of compact and connected three dimensional
submanifolds of $M$, a sequence $\tau_{i}>0$ that converges to zero and a
sequence of $\tau_{i}$-approximations $\mathfrak{p}_{i}:\mathcal{N}%
_{i}\rightarrow Y$ which induces a structure of locally trivial bundle on
$\mathcal{N}_{i}$ whose fibers are either two dimensional spheres or tori.

\begin{lemma}
\label{fibras sao toros}\textit{The fibers of }$\mathfrak{p}_{i}%
:\mathcal{N}_{i}\rightarrow Y$ \textit{are tori.}
\end{lemma}

\noindent\textbf{Proof of Lemma} (\ref{fibras sao toros})\textbf{\ :} Since
the fibers are either spheres or tori, it suffices to nule out the first
possibility. Looking for a contradiction, let us suppose that the fibers are
spheres. Given $z_{1}\in Z-\mathcal{B}$ and a positive constant $\delta\ll
diam_{Z}\left(  K\right)  $, consider a sequence of points $q_{i}\in
N_{i}-\Sigma$ converging (in the Hausdorff-Gromov sense) to $z_{1}$ and such
that $d_{HG}\left(  q_{i},z_{1}\right)  <\delta$ for all $i\in\mathbb{N}%
$. 

\bigskip

\noindent\textbf{Claim }: \textit{There exists }$i_{0}\in N$\textit{\ such
that, for every }$i>i_{0}$\textit{, we can find a homotopically nontrivial
loop (in }$M-\Sigma$\textit{) }$\gamma_{i}$\textit{\ based on }$q_{i}%
$\textit{\ whose length does not exceed }$\delta$\textit{. }

\noindent\textbf{Proof of Claim : }Consider the loops based on $q_{i}$ which
are constituted by two minimizing geodesic segments with same ends and having
equal lengths bounded by $\frac{\delta}{2}$. Note that these loops are always
homotopically nontrivial.

To say that a point $q_{i}$ is not a base point for a loop as above is
equivalent to say that its injectivity radius not less than $\frac{\delta}{2}%
$. This is a contradiction with the collapsing assumption and therefore the
claim follows.\hfill$\diamond$ 

\bigskip

Now fix $i>i_{0}$. Consider a boundary sphere of $\mathcal{N}_{i}$ associates
with an unbounded component of $Z-K$. By construction, this sphere separates
$M$ in two components: one of these components contains the singular link
$\Sigma_{0}\neq\emptyset$ (recall that $\Sigma_{1}\subset\Sigma_{0}$ by
hypothesis) and the other contains a homotopically nontrivial loop as in the
claim. Since $M-\Sigma$ is irreducible, this cannot happen and therefore the
fibers must be solid tori. This establishes the lemma.\hfill$\diamond$ 

As in the two dimensional case, the component(s) of $N_{i}-\mathcal{N}_{i}$
associated with unbounded components of $Z-Y$ are solid tori. When $Z$ is
homeomorphic to $\left[  0,\infty\right)  $, Shioya and Yamaguchi showed (see
\cite[Theorem 0.5 and Table 1]{SY}) that the component of $N_{i}%
-\mathcal{N}_{i}$ associated with the bounded component of $Z-Y$ is
homeomorphic to either a solid torus or to $M\ddot{o}\widetilde{\times}S^{1}$.
It follows that being irreducible $M$ is according homeomorphic to a lens space or a prism
manifold.\hfill
\end{proof}

\subsubsection{Case where $Z$ has dimension $0$}

This case was treated in \cite{SY}. For the convenience of the reader, we
state here their result:

\begin{theorem}
[collapsing - dimension 0]\label{effondrement en dim 0}Suppose the existence
of a point $p\in\Sigma_{1}$ such that the sequence $\left(  M_{i},p\right)  $
converges (in the Hausdorff-Gromov sense) to a zero dimensional pointed
Alexandrov space $\left(  Z,z_{0}\right)  $ and also that $\sup\left\{
R_{i}\left(  \Sigma_{1}\right)  \;;\;i\in\mathbb{N}\right\}  <\infty$. Then
$M$ is an euclidean, spherical or $Nil$ manifold.
\end{theorem}

\section{Applications}

\subsection{Proof of Corollary (\ref{corolario da conjectura})}

\begin{proof}
Let $M_{\alpha_{i}}$ be a convergent sequence as in the statement of the
corollary. Since $M$ is a hyperbolic manifold, it is not Seifert fibered or a
$Sol$ manifold. Therefore, in the presence of the hypothesis
(\ref{controlecomprimento2}), is a consequence of the Theorem
(\ref{teo principal introducao}) that the sequence $M_{\alpha_{i}}$ cannot
collapses.\hfill
\end{proof}

\subsection{Volume of Representations}

\qquad Recall that $M$ denotes a closed, orientable (not necessarily
irreducible) differential manifold of dimension $3$ and that $\Sigma
=\Sigma_{1}\sqcup\ldots\sqcup\Sigma_{k}$ denotes an embedded link in $M$. Let
us recall the definition of the volume of representations and some of its
properties. For the interested reader, we refer to \cite{Fra} and \cite{Dun}
for details.

Let $U=U_{1}\sqcup\ldots\sqcup U_{k}$ be an open neighborhood of $\Sigma$,
where each $U_{i}$ is a neighborhood of $\Sigma_{i}$ homeomorphic to a solid
torus. Consider product structures $\left(  T^{2}\times\lbrack0,\infty
)\right)  _{i}$ on the open sets $\mathcal{U}_{i}=U_{i}-\Sigma$ called
cuspidal ends of $M-\Sigma$. We will denote the universal cover of $M-\Sigma$
by $\pi:\widetilde{M-\Sigma}\rightarrow M-\Sigma$. The connected components of
$\pi^{-1}\left(  \mathcal{U}_{i}\right)  $ are denoted by $\mathcal{V}%
_{ij}\thickapprox\left(  \pi^{-1}\left(  T^{2}\right)  \times\lbrack
0,\infty)\right)  _{ij}$.

\begin{definition}
Given a representation $\rho:\pi_{1}(M-\Sigma)\rightarrow PSL_{2}\left(
\mathbb{C}\right)  $, a positive piecewise differentiable map $d:\Sigma
\rightarrow M-\mathbb{H}^{3}$ is called a pseudo-developing map for $\rho$ if
it verify:

\begin{itemize}
\item[i.] (equivariance) $d\circ\gamma=\rho\left(  \gamma\right)  \circ d$,
for all $\gamma\in\pi_{1}(M-\Sigma)$

\item[ii.] There is $r_{o}\in\lbrack0,\infty)$ such that, for every $\left(
p,r_{0}\right)  _{ij}\in\mathcal{V}_{ij}$, the curve%
\[
t\in\lbrack r_{o},\infty)\longmapsto d\left(  \left(  p,t\right)
_{ij}\right)  \in\mathbb{H}^{3}%
\]
is a geodesic of $\mathbb{H}^{3}$ parameterized by arc length. In addition,
there are points $\zeta_{ij}\in\partial\mathbb{H}^{3}$ (depending only on the
indexes $i$ and $j$), such that
\[
\lim_{t\rightarrow\infty}d\left(  \left(  p,t\right)  _{ij}\right)
=\zeta_{ij}\in\partial\mathbb{H}^{3}\text{.}%
\]

\end{itemize}

Given a pseudo-develloping map for $\rho$, let $\omega_{d}$ be the unique
differential form of degree $3$ on $M-\Sigma$ verifying%
\[
\left.  \pi\right.  ^{\ast}\omega_{d}=d^{\ast}V_{\mathbb{H}^{3}}\text{,}%
\]
where $V_{\mathbb{H}^{3}}$ denotes the volume form of $\mathbb{H}^{3}$. The
volume of $\rho$ is defined by%
\[
Vol_{M-\Sigma}\left(  \rho\right)  =%
{\displaystyle\int\nolimits_{M-\Sigma}}
\omega_{d}=%
{\displaystyle\int\nolimits_{R}}
d^{\ast}V_{\mathbb{H}^{3}}<\infty\text{ ,}%
\]
where $R\subset\widetilde{M-\Sigma}$ is a fundamental region for the action of
$\pi_{1}\left(  M-\Sigma\right)  $ on $\widetilde{M-\Sigma}$.
\end{definition}

Given a representation $\rho:\pi_{1}(M-\Sigma)\rightarrow PSL_{2}\left(
\mathbb{C}\right)  $, note that a pseudo-developing map for $\rho$ always
exists. Moreover, the volume of $\rho$ does not depend on the chosen
pseudo-developing map. The following result presents some properties of the
volume function%

\[
Vol_{M-\Sigma}:R(M-\Sigma)\longrightarrow\mathbb{R}^{+}\text{,}%
\]
where $R(M-\Sigma)$ denotes the set of all representations of $\pi
_{1}(M-\Sigma)$ in $PSL_{2}\left(  \mathbb{C}\right)  $ with its canonical
structure of algebraic variety.

\begin{proposition}
\label{propriedades vol (cusps toroidales)}The function $Vol_{M-\Sigma}$ is
continuous and verifies the following properties:

\begin{itemize}
\item[i.] $Vol_{M-\Sigma}\left(  \rho\right)  =0$, for any reducible
representation $\rho\in R(M-\Sigma)$. Recall that a representation $\rho\in
R(M-\Sigma)$ is called reducible if the group $\rho(\pi_{1}\left(
M-\Sigma\right)  )\subset PSL_{2}(\mathbb{C})$ fixes a point in $\partial
H^{3}$,

\item[ii.] The function $Vol_{M-\Sigma}$ factorizes to the quotient $R\left(
M-\Sigma\right)  /PSL_{2}\left(  \mathbb{C}\right)  $. That is,%
\[
Vol_{M-\Sigma}(\rho)=Vol_{M-\Sigma}(A.\rho)\text{,}%
\]
for every representation $\rho\in R(M-\Sigma)$ and every element $A\in
PSL_{2}(\mathbb{C})$.

\item[iii.] Consider a closed, orientable, differentiable manifold $N$ of
dimension $3$. Let $\Omega$ be an embedded link in $N$. If $f:\left(
N,\Omega\right)  \rightarrow\left(  M,\Sigma\right)  $ is a diffeomorphism,
then%
\[
Vol_{M-\Sigma}(\rho)=Vol_{N-\Omega}\left(  \rho\circ f_{\ast}\right)  \text{,}%
\]
for every representation $\rho\in R\left(  M-\Sigma\right)  $.
\end{itemize}
\end{proposition}

The following propositions recall two essential properties of the volume
function. The first concerns the operation of Dehn filling. The second one
establishes the relationship between the Riemannian volume of a hyperbolic
cone-manifold and the volume of its holonomy.

\begin{proposition}
\label{egalite du vol de rep apres dehn filling}Given a compact, orientable,
differentiable $3$-manifold $W$ with tori boundary, let $M$ be a closed
$3$-manifold obtained from $W$ by Dehn filling on its tori boundary
components. For every representation $\rho\in R(M)$ we have that%
\[
Vol_{M}(\rho)=Vol_{M-\Sigma}\left(  \rho\circ i_{\ast}\right)  \text{,}%
\]
where $\Sigma$ denotes the union of souls cores of the tori added to $W$ and
\[
i_{\ast}:\pi_{1}\left(  M-\Sigma\right)  \rightarrow\pi_{1}\left(  M\right)
\]
is the canonical map induced by the inclusion $i:M-\Sigma\hookrightarrow M$.
\end{proposition}

\begin{proposition}
\label{vol rep = vol riem de var conique}Let $\mathcal{M}$ be a hyperbolic
cone-manifold with topological type $\left(  M,\Sigma\right)  $. If
$\rho_{\mathcal{M}}\in R(M-\Sigma)$ is a holonomy representation for
$\mathcal{M}$, then%
\[
Vol(M-\Sigma)=Vol_{M-\Sigma}(\rho_{\mathcal{M}})\text{,}%
\]
where $Vol(M-\Sigma)$ is the Riemannian volume of $M-\Sigma$.
\end{proposition}

Suppose from now on that $M$ is irreducible. Given a sequence $M_{i}$ of
hyperbolic cone-manifolds with topological type $\left(  M,\Sigma\right)  $,
let $\rho_{i}$ be an associated sequence of holonomy representations. Suppose
also that the sequence $\rho_{i}$ converges to a representation $\rho_{\infty
}\in R\left(  M-\Sigma\right)  $.

As a consequence of the above propositions, we have two applications of the
Theorem (\ref{teo principal introducao}) concerning the volume of holonomy
representations of a hyperbolic cone-manifold.

\begin{corollary}
If $Vol_{M-\Sigma}\left(  \rho_{\infty}\right)  =0$ (in particular, if
$\rho_{\infty}$ is reducible), then $M$ is Seifert fibered or a Sol manifold.
\end{corollary}

\begin{proof}
According to the proposition (\ref{vol rep = vol riem de var conique}), we
have%
\[
Vol\left(  M_{i}\right)  =Vol_{M-\Sigma}\left(  \rho_{i}\right)  \text{,}%
\]
for every $i\in\mathbb{N}$. Since the function $Vol_{M-\Sigma}$ is continuous
we also have%
\[
\lim\limits_{i\rightarrow\infty}Vol\left(  M_{i}\right)  =\lim
\limits_{i\rightarrow\infty}Vol_{M-\Sigma}\left(  \rho_{i}\right)
=Vol_{M-\Sigma}\left(  \rho_{\infty}\right)  =0\text{.}%
\]
As a consequence, the sequence $M_{i}$ necessarily collapses.

Recall that the convergence hypothesis implies the condition
(\ref{controledocomprimento}). The result is now a direct consequence of the
Theorem (\ref{teo principal introducao}).\hfill
\end{proof}

\begin{corollary}
Denote by $\alpha_{ij}$ the cone angle of the component $\Sigma_{j}$ in
$M_{i}$. If $M$ has zero simplicial volume and%
\[
\lim\limits_{i\rightarrow\infty}\alpha_{ij}=2\pi\qquad j\in\left\{
1,\ldots,k\right\}  \text{,}%
\]
then $M$ is Seifert fibered or a Sol manifold.
\end{corollary}

\begin{proof}
Since the cone angles converge to $2\pi$, we have%
\[
\rho_{\infty}(\left[  \mu\right]  _{\pi_{1}\left(  M-\Sigma\right)
})=1_{PSL_{2}(\mathbb{C)}}\text{ ,}%
\]
for all meridian $%
\mu
$ of $\Sigma$. Then the representation $\rho_{\infty}\in R\left(
M-\Sigma\right)  $ factorizes to a representation $\psi_{\infty}\in R\left(
M\right)  $. More specifically%
\[
\rho_{\infty}=\psi_{\infty}\circ i_{\ast}\text{,}%
\]
where $i_{\ast}$ is the map induced by the inclusion $i:M-\Sigma
\hookrightarrow M$. Since the simplicial volume of $M$ is zero, we have (see
\cite{Fra})
\[
Vol_{M}\left(  \psi_{\infty}\right)  =0
\]
and then (Proposition \ref{egalite du vol de rep apres dehn filling})%
\[
Vol_{M-\Sigma}\left(  \rho_{\infty}\right)  =Vol_{M-\Sigma}\left(
\psi_{\infty}\circ i_{\ast}\right)  =Vol_{M}\left(  \psi_{\infty}\right)
=0\text{.}%
\]
The assertion follows from the preceding corollary.\hfill
\end{proof}

\bibliographystyle{amsplain}

\end{document}